\documentclass[12pt]{amsart}
\usepackage{amsmath,amsthm}
\numberwithin{equation}{section}
\newtheorem{thm}[equation]{Theorem}
\newtheorem{cor}[equation]{Corollary}
\newtheorem{lm}[equation]{Lemma}

\newtheorem{clm}[equation]{Claim}
\theoremstyle{definition}

\theoremstyle{remark}

\DeclareMathOperator{\BZ}{\mathbb{Z}} % integers
\DeclareMathOperator{\BE}{\mathbb{E}} % integral quaternions
\DeclareMathOperator{\BL}{\mathbb{L}} % quaternions with int. coefficients.
\newcommand{\vep}{\varepsilon}
\DeclareMathOperator{\Norm}{N}        % norm

\newcommand{\set}[2]{\{#1\mid\nobreak\text{#2}\}} % { x | ... }

\newcommand{\konj}{\overline}

\newcommand{\hphm}{\hphantom{-}}
\begin{document}
\date{Int4vec8, 22 April, 2011}
\title{Cubes of integral vectors in dimension four}
\author{Emil W. Kiss, P\'eter Kutas}
\address[Emil W. Kiss]{E{\"o}tv{\"o}s University\\
  Department of Algebra and Number Theory\\
  1117 Budapest, P{\'a}zm{\'a}ny P{\'e}ter
  s{\'e}t{\'a}ny 1/c\\
  Hungary}
\email[Emil W. Kiss]{ewkiss@math.elte.hu}
\email[P\'eter Kutas]{kutasp@gmail.com}

\thanks{Supported by Hungarian Nat.\ Sci.\
  Found.\ (OTKA) Grant No.\ NK72523.}



\begin{abstract}
  A system of $m$ nonzero vectors in $\BZ^n$ is called an
  $m$-icube if they are pairwise orthogonal and have the
  same length. The paper describes $m$-icubes in $\BZ^4$ for
  $2\le m\le 4$ using Hurwitz integral quaternions, counts
  the number of them with given edge length, and proves that
  unlimited extension is possible in~$\BZ^4$.
\end{abstract}

\maketitle

\parskip=0pt plus 2.5pt

\section{Introduction and main results}
\label{sec_intr}

Two vectors are called \emph{twins} if they are orthogonal,
and have the same length. An \emph{$m$-icube} in $\BZ^n$ is
a sequence $(v_1,\ldots,v_m)$ of nonzero vectors in~$\BZ^n$
that are twins pairwise. The common \emph{length} of the
vectors~$v_\ell$ is the \emph{edge length} of the
icube. By the \emph{norm} of $v_\ell$ we mean the square of
its length. The main object of this paper is to study how
icubes can be {\it constructed, extended} and {\it
  counted}. The paper \cite{GKMS08} investigates these
questions extensively in~$\BZ^3$, using the number theory of
quaternions.

For a trivial example, if the dimension is even, then every
vector $(a_1,\ldots,a_{n})$ has a twin, namely
$(a_2,-a_1,a_4,-a_3,\ldots,a_n,-a_{n-1})$. Similarly, the
rows of the matrix
\[
\begin{pmatrix}
 a&\hphm b&\hphm c&\hphm d&\hphm e&\hphm f&\hphm g&\hphm h \\
 b&-a&\hphm d&-c&\hphm f&-e&-h&\hphm g \\
 c&-d&-a&\hphm b&\hphm g&\hphm h&-e&-f \\
 d&\hphm c&-b&-a&\hphm h&-g&\hphm f&-e \\
 e&-f&-g&-h&-a&\hphm b&\hphm c&\hphm d \\
 f&\hphm e&-h&\hphm g&-b&-a&-d&\hphm c \\
 g&\hphm h&\hphm e&-f&-c&\hphm d&-a&-b \\
 h&-g&\hphm f&\hphm e&-d&-c&\hphm b&-a 
\end{pmatrix}
\]
form an $8$-icube, proving that every $8$-dimensional
integral vector can be extended to an $8$-icube. The above
matrix comes from the multiplication table of
Cayley-numbers. The $4\times 4$ minor in the upper left
corner yields a $4$-icube in dimension~$4$, extending an
arbitrary element of $\BZ^4$.

Classical results of Hurwitz \cite{H23} and Radon \cite{R22}
show, however, that a similar ``permutational'' extension is
possible only in dimensions $1$, $2$, $4$ and $8$ (an
interesting approach using extraspecial $2$-groups is given
by Eckmann in \cite{E43}). To prove further extension
theorems we have to explore the number-theoretic structure
of the components of the vectors. An example for this type
of argument is the Euler-matrix
\[
\begin{pmatrix}
m^2+n^2-p^2-q^2&-2mq+2np&2mp+2nq\\
2mq+2np&m^2-n^2+p^2-q^2&-2mn+2pq\\
-2mp+2nq&2mn+2pq&m^2-n^2-p^2+q^2
\end{pmatrix}\,,
\]
which is a ``typical'' $3$-icube in dimension $3$ (see
\cite{Sar61} and \cite{GKMS08}). We start with an extension
theorem that generalizes Corollary~5.11 of \cite{GKMS08}.

\begin{thm}\label{main_extension_gen}
Let $(v_1,\ldots,v_{n-1})$ be an $n-1$-icube in $\BZ^n$,
where $n\ge 2$. If $n$ is even, then this icube can be
extended to an $n$-icube. If $n$ is odd, then such an
extension is possible if and only if the common length of
the vectors $v_\ell$ is an integer.
\end{thm}

Note that this extending vector, if exists, is obviously
unique up to sign.

\begin{proof} Let $N$ denote the edge norm
  of~$(v_1,\ldots,v_{n-1})$. By Proposition~1.3 of
  \cite{GKMS08}, if $n$ is odd, then the edge length of any
  $n$-icube in $\BZ^n$ is an integer. Therefore an extension
  is only possible if $n$ is even or if $N$ is a
  square.

  Define $L$ to be the $n\times (n-1)$ matrix whose columns
  are $v_1,\ldots,v_{n-1}$. Then $L^TL=NI_{n-1}$ (where
  $I_{n-1}$ denotes the identity matrix). The Cauchy-Binet
  formula therefore implies that
\[
\det(L_1)^2+\dots +\det(L_n)^2=N^{n-1}\,,
\]
where $L_i$ is the minor of~$L$ obtained by deleting the
$i$-th row.

Let $M_i=(-1)^{n+i}\det(L_i)$. Add a last column to $L$
whose entries are $M_i/N^{(n-2)/2}$, and denote the
resulting matrix by~$K$. Then the columns of $L$ are
pairwise orthogonal by the Laplace expansion theorem for
determinants. The displayed formula above shows that
$K^TK=NI_n$. This implies that $KK^T=NI_n$. Denote the rows
of~$L$ by~$s_i$. We get that the scalar product of $s_i$ by
itself, which is an integer, equals
$N-M_i^2/N^{n-2}$. Therefore if $n$ is even or if $N$ is a
square, then $N^{(n-2)/2}$ divides $M_i$, and the last
column of $K$ consists of integers.
\end{proof}

Here are the main results of this paper.

\begin{thm}\label{main_extension}
Every $m$-icube in $\BZ^4$ can be extended to a $4$-icube
for $1\le m\le 3$.
\end{thm}

Of course, the only nontrivial case occurs when $m=2$,
according to the statements above. The proof is found at the
end of Section~\ref{sec_twin}.

The following result, proved in Section~\ref{sec_count},
counts the number of $m$-icubes in~$\BZ^4$. Denote by
$f_m(N)$ the number of $m$-icubes with edge norm~$N$ (that
is, edge length~$\sqrt{N}$) in $\BZ^4$. A famous theorem by
Jacobi provides the value of $f_1(N)$, we include it for
comparison. Let $c_m=24\cdot2^m/(4-m)!$, thus $c_1=8$,
$c_2=48$, $c_3=192$ and $c_4=384$. Furthermore, if $p$ is a
(positive) odd prime and $k\ge 1$, then define
\[
g(p^k)=\frac{(k+1)p^k(p^2-1)-2(p^{k+1}-1)}{(p-1)^2}\,.
\]

\begin{thm}\label{main_count}
  Let $g_m(N)=f_m(N)/c_m$. Then $g_m$ is a multiplicative
  function for every $1\le m\le 4$, whose value on prime
  powers is given by the following formulae, where $p$ is an
  odd prime and $k\ge 1$.
\begin{enumerate}
\item[$(1)$] $g_m(2^k)=3$ for every $k\ge 1$.
\item[$(2)$] $g_1(p^k)=\sigma(p^k)=(p^{k+1}-1)/(p-1)$ (this
  is Jacobi's classical result) and
  $g_3(p^k)=g_4(p^k)=g(p^k)$ (the function defined before
  the theorem).
\item[$(3)$] If $p\equiv 3~(4)$, then $g_2(p^k)=g(p^k)$. If
  $p\equiv 1~(4)$, then $g_2(p^k)=(k+1)p^k$.
\end{enumerate}
In particular, we have that $f_4(N)=2f_3(N)$.
\end{thm}

The proofs are based on a representation theorem of icubes
using Hurwitz integral quaternions (see
Theorems~\ref{twin_decomp_thm}, \ref{twin_decomp_odd},
\ref{icube_unique} and~\ref{icube_unique2}).

\section{Integral quaternions}
\label{sec_quat}

We review some properties of integral quaternions. The
general references are \cite{CS03}, \cite{H19} and
\cite{HW79}, but we ask the reader to browse Section~2 of
\cite{GKMS08} for background, as we shall use the notation
and the results introduced there. The \emph{norm} of
$\alpha=a+bi+cj+dk$ is $\Norm(\alpha)=a^2+b^2+c^2+d^2$. This
$\alpha$ is a \emph{Hurwitz} integral quaternion if
$2a,2b,2c,2d$ are all integers of the same parity. Hurwitz
integral quaternions form a left Euclidean ring~$\BE$. The
ring of quaternions with integral coefficients is denoted
by~$\BL$ (these are the \emph{Lipschitz} integral
quaternions). The sign $\alpha\mid\beta$ means:
$\alpha$~divides~$\beta$ on the left in~$\BE$. The ring
$\BE$ has $24$ units. Every element $\alpha$ of $\BE$ has a
left associate in $\BL$ and a right associate in~$\BL$.

\begin{thm}[\cite{GKMS08}, Theorem 2.7; see Theorem 377 of
  \cite{HW79} and the note after the proof of Theorem~3 in
  Section~5.3 of~\cite{CS03}]\label{quat_irred}
  An integral quaternion is irreducible in the ring~$\BE$ if
  and only if its norm is a prime in $\BZ$. The only
  elements of\/ $\BE$ whose norm is $2$ are $1+i$ and its
  left associates. If $p>2$ is a prime in $\BZ$, then there
  exist exactly $24(p+1)$ integral quaternions whose norm is
  $p$.
\end{thm}

\begin{lm}[\cite{GKMS08}, Lemma 2.5]\label{quat_lnko}
  Suppose that $\alpha\in\BE$ and $p\in\BZ$ is a prime such
  that $p\mid\Norm(\alpha)$ but $p$~does not divide
  $\alpha$. Then $\alpha$ can be written as $\pi\alpha'$
  where $\Norm(\pi)=p$, and this $\pi$ is uniquely
  determined up to right association.
\end{lm}

\begin{lm}[\cite{GKMS08}, Lemma 2.6]\label{twin_primeprop}
  Suppose that $\theta,\eta,\pi\in\BE$ such that
  $\Norm(\pi)=p$ is a prime in~$\BZ$. If $\pi\mid\theta$,
  $p\mid\konj{\theta}\eta$ but $p$ does not divide $\theta$,
  then $\pi\mid\eta$.
\end{lm}

We shall reduce questions to quaternions having odd norm,
using the following assertion.

\begin{clm}\label{twin_even}
Let $\alpha=a+bi+cj+dk\in\BL$.
\begin{enumerate}
\item[$(1)$] There exists an element $\beta\in\BL$ such that
  $\alpha=(1+i)\beta$ if and only if $a\equiv b~(2)$ and
  $c\equiv d~(2)$. The analogous statements hold for $1+j$
  and $1+k$.
\item[$(2)$] If\/ $8\mid\Norm(\alpha)$, then each
  coefficient of $\alpha$ is even.
\item[$(3)$] If\/ $\Norm(\alpha)\equiv 4~(8)$, then
  $\alpha=(1+i)\beta$ for some $\beta\in\BL$.
\item[$(4)$] If\/ $\Norm(\alpha)\equiv 2~(4)$, then there is
  exactly one element $\eta\in\{1+i,1+j,1+k\}$ such that
  $\alpha=\eta\beta$ for some $\beta\in\BL$.
\end{enumerate}
\end{clm}

\begin{proof}
$(1)$ can be shown by direct calculation. Since $m^2\equiv
  1~(8)$ for every odd integer~$m$, we see that
  $\Norm(\alpha)$ is divisible by $8$ if and only if $a$,
  $b$, $c$, $d$ are all even, so $(2)$ holds. By the same
  argument, if $4\mid\Norm(\alpha)$, then $a$, $b$, $c$, $d$
  are all even, or are all odd. In the first case we have
  $(3)$, since $2=(1+i)(1-i)$. In the second case $(3)$~also
  holds by~$(1)$. Now suppose that $\Norm(\alpha)\equiv
  2~(4)$. Then two numbers of $a$, $b$, $c$, $d$ are even
  and two are odd. If $a\equiv b~(2)$, then $c\equiv d~(2)$,
  so $(1)$ shows that $(1+i)$ can be pulled out
  from~$\alpha$, but $1+j$ and $1+k$ cannot.
\end{proof}

Next we investigate quaternions with integral coefficients
having odd norm. Let $K=\{1,i,j,k\}$, and write a general
quaternion $\alpha\in\BL$ as $a_1+a_ii+a_jj+a_kk$. For $g\in
K$ define
\[
S_g=\set{a_1+a_ii+a_jj+a_kk\in\BL}{$a_g\not\equiv a_h~(2)$
  for every $h\ne g$, where $h\in K$}\,.
\]
So for example the elements of $S_i$ are those where the
coefficient of $i$ is odd and the other coefficients are
even, or vice versa. Let $*$ denote the Klein-group
multiplication on~$K$ (which is quaternion-multiplication,
but disregards the signs). Call two nonzero quaternions
\emph{twins} if so are the vectors formed by their
coefficients. The following claim summarizes well-known,
easy facts.

\begin{clm}\label{Sg_elementary}
Let $\alpha,\beta\in\BL$ with odd norm and
$2\sigma=1+i+j+k$.
\begin{enumerate}
\item[$(1)$] Both $\alpha$ and $\beta$ belong to exactly one
  of the sets~$S_g$. If they are twins, then they cannot
  belong to the same~$S_g$.
\item[$(2)$] If\/ $\Norm(\alpha)\equiv 1~(4)$ then $\alpha\in
  S_g$ if and only if $\alpha\equiv g~(2)$ in $\BL$ if and
  only if $\alpha\equiv g~(2)$ in $\BE$.
\item[$(3)$] If\/ $\Norm(\alpha)\equiv 3~(4)$, then
  $\alpha\in S_g$ if and only if $\alpha\equiv
  2\sigma-g~(2)$ in $\BL$ if and only if $\alpha\equiv
  2\sigma-g~(2)$ in $\BE$.
\item[$(4)$] If\/ $\alpha\in S_g$ and $\beta\in S_h$, then
  $\alpha\beta\in S_{g*h}$.
\item[$(5)$] If $\alpha\in S_1$ and $\gamma\in\BE$, then
  $\gamma\in S_g\iff \alpha\gamma\in S_g\iff \gamma\alpha\in
  S_g$.
\end{enumerate}
\end{clm}

\begin{proof}
  If $\Norm(\alpha)\equiv 1~(4)$, then $\alpha$ has exactly
  one odd component. If $\Norm(\alpha)\equiv 3~(4)$, then
  $\alpha$ has exactly one even component. Suppose that
  $\alpha$ and $\beta$ are twins in $S_g$. Then
  $\Norm(\alpha)=\Norm(\beta)\equiv 1~(4)$ implies that the
  scalar product of the corresponding vectors is congruent
  to $1$ modulo~$2$, which is impossible, since they are
  orthogonal. If $\Norm(\alpha)=\Norm(\beta)\equiv 3~(4)$,
  then this scalar product is congruent to $3$ modulo~$2$,
  also a contradiction. This shows~$(1)$. The proofs of
  $(2)-(4)$ are left to the reader.

  Suppose that $\alpha\in S_1$ and $\alpha\gamma=\delta\in
  S_g$. Then $\Norm(\alpha)\gamma
  =\konj{\alpha}\delta\in\BL$. Since $\Norm(\alpha)$ is odd,
  this implies that $\gamma\in\BL$. The norm of $\gamma$ is
  odd, so $(5)$ follows from~$(4)$.
\end{proof}

Call a quaternion $\alpha$ \emph{primary} if $\alpha\in S_1$
and $a_1+a_i+a_j+a_k\equiv 1~(4)$. Obviously, if $\alpha\in
S_1$, then exactly one of $\alpha$ and $-\alpha$ is primary.

\begin{clm}\label{Sg_elementary2}
The following hold.
\begin{enumerate}
\item[$(1)$] If $\gamma\in\BE$ has odd norm, then $\gamma$
  has exactly one primary left associate, and exactly one
  primary right associate.
\item[$(2)$] The primary quaternions form a semigroup under
  multiplication.  Moreover, if the (left or right) quotient
  of two primary quaternions is in $\BE$, then it is also
  primary.
\item[$(3)$] Let $\alpha$ be a primary quaternion and
  $\vep\in\BE$ a unit. Then $\vep\alpha\in\BL$ (or
  $\alpha\vep\in\BL$) if and only if $\vep\in Q=\{\pm 1, \pm
  i, \pm j, \pm k\}$.
\end{enumerate}
\end{clm}

\begin{proof}
Statement $(3)$ clearly follows from
Claim~\ref{Sg_elementary}~$(5)$. The rest of the proof is
left to the reader.
\end{proof}

We close this section with two counting results. Call a
quaternion with integral coefficients \emph{primitive}, if
its coefficients are relatively prime.

\begin{clm}[Jacobi]\label{twin_primary_count}
  Let $N>1$ be odd. Then the number of primary primitive
  quaternions with norm $N$ is $h(N)=N\prod_p
  \big(1+(1/p)\big)$, where $p$ runs over the prime divisors
  of $N$.
\end{clm}

A \emph{pure} quaternion is one with real part zero.

\begin{lm}[see \cite{GKMS08}, Theorem 4.2]\label{3repr}
  Let $\theta\in\BE$ be a primitive pure quaternion whose
  norm is a square. Then $\theta$ can be written as $\gamma
  \,i\,\konj{\gamma}$ for some $\gamma\in\BE$. Here $\gamma$
  is uniquely determined in the sense that any two such
  elements $\gamma$ are right associates via a unit in
  $\{1,-1,i,-i\}$.
\end{lm}

\begin{clm}\label{twin_count2}
  Let $N>1$ be odd. Then the number of primary quaternions
  $\gamma$ with norm~$N$ such that
  $\gamma\,i\,\konj{\gamma}$ is primitive is $q(N)=N\prod_p
  \big(1-(s_p/p)\big)$, where $p$ runs over the prime
  divisors of~$N$ and $s_p\in\{1,-1\}$ is congruent to $p$
  modulo~$4$.
\end{clm}

(In other words, $s_p=(-1)^{(p-1)/2}=(-1/p)$ as a
Legendre-symbol).

\begin{proof}
Theorem~4.8 in \cite{GKMS08} implies that the number of
primitive vectors $(x,y,z)$ with norm $N^2$ is $6q(N)$. By
Lemma~\ref{3repr}, the quaternions corresponding to such
vectors can be written as $\theta=\gamma \,i\,\konj{\gamma}$
for some $\gamma\in\BE$. Conjugacy with the units in~$\BE$
yields an equivalence relation on the set of all such
elements $\theta$. The fact that $\theta$ is primitive, but
not a unit implies that at least two of its components are
nonzero. Therefore each conjugacy class has $12$ elements
(the stabilizer is just $\{1,-1\}$ in each case). It is
sufficient to show that exactly two of these conjugates can
be written in the form $\gamma \,i\,\konj{\gamma}$ such that
$\gamma$ is primary.

If $\theta=\gamma \,i\,\konj{\gamma}$, then
Claim~\ref{Sg_elementary2} shows that $\gamma=\vep\alpha$,
where $\vep$ is a unit and $\alpha$ is primary. Therefore
$\theta$ has a conjugate of the required form, namely
$\alpha \,i\,\konj{\alpha}$. Exactly one of $\beta=\pm
i\alpha\konj{i}$ is primary, and $\beta
\,i\,\konj{\beta}=i\theta\konj{i}$ is a conjugate of
$\theta$ that is different from~$\theta$.

Conversely, suppose that $\theta$ has two conjugates
$\theta_1=\gamma_1 \,i\,\konj{\gamma_1}$ and
$\theta_2=\gamma_2 \,i\,\konj{\gamma_2}$ such that
$\gamma_1$ and $\gamma_2$ are primary. Thus
$\theta_2=\vep\theta_1\konj{\vep}$ for some unit~$\vep$. By
the uniqueness statement of Lemma~\ref{3repr},
$\vep\gamma_1=\gamma_2\rho$ for some
$\rho\in\{1,-1,i,-i\}$. Claim~\ref{Sg_elementary2} shows
that $\vep\in Q$, and Claim~\ref{Sg_elementary}~$(5)$ gives
that $\vep=\pm\rho$.  If $\vep=\pm 1$, then
$\theta_1=\theta_2$. If $\vep=\pm i$, then
$\theta_2=i\theta_1\konj{i}$.
\end{proof}

There is an alternative argument for the previous statement:
the reader may go through the proof of Theorem~4.8 in
\cite{GKMS08}, and modify it in such a way that only primary
prime factors are used when building~$\gamma$.

\section{Construction and extension}
\label{sec_twin}

We shall speak about $m$-icubes $(\alpha_1,\ldots,\alpha_m)$
in $\BE$ and in $\BL$, meaning that this is a sequence of
simultaneous twins such that each $\alpha_\ell$ lies in
$\BE$ or in $\BL$, respectively.

\begin{lm}\label{twin_char}
The quaternions $\alpha$ and $\beta$ are twins if and only
if their norms are equal, and
$\konj{\alpha}\beta=-\konj{\beta}\alpha$ (or equivalently,
$\alpha\konj{\beta}=-\beta\konj{\alpha}$) holds.
\end{lm}

\begin{proof}
If $\alpha=a_1+a_ii+a_jj+a_kk$ and
$\beta=b_1+b_ii+b_jj+b_kk$, then the real part of
$\alpha\beta$ is $a_1b_1-a_ib_i-a_jb_j-a_kb_k$. Therefore
the vectors corresponding to $\alpha$ and $\beta$ are
orthogonal if and only if the real part of
$\alpha\konj{\beta}$ is zero (if and only if the real part
of $\konj{\alpha}\beta$ is zero). However, the real part of
a quaternion is zero if and only if its conjugate is its
negative.
\end{proof}

\begin{cor}\label{twin_char_cor}
If $\gamma\ne 0$, then $\alpha$ and $\beta$ are twins if and
only if $\alpha\gamma$ and $\beta\gamma$ are twins if and
only if $\gamma\alpha$ and $\gamma\beta$ are twins.\qed
\end{cor}

\begin{lm}\label{twin_param}
Let $\alpha,\beta\in\BE$ be twins and $p\in\BZ$ be a prime
dividing $\Norm(\alpha)=\Norm(\beta)$. Then there exists a
quaternion $\pi\in\BE$ with norm~$p$ such that either $\pi$
divides both $\alpha$ and $\beta$ on the left, or $\pi$
divides both $\alpha$ and $\beta$ on the right. If
$\beta\konj{\alpha}$ is divisible by $p$, then the second
case surely holds.
\end{lm}

\begin{proof}
If $p\mid \alpha$, then every $\pi\in\BE$ with norm~$p$
divides $\alpha$ both on the left and on the right, since
$p=\pi\konj{\pi}=\konj{\pi}\pi$ (and such an element exists
by Theorem~\ref{quat_irred}). If $\alpha$ is not divisible
by $p$, then Lemma~\ref{quat_lnko} yields a left
divisor~$\pi_1\in\BE$ with norm~$p$. Applying this lemma to
$\konj{\alpha}$ we get a right divisor $\pi_2$ of $\alpha$
with norm~$p$. Similarly, $\beta$ has a left divisor $\pi_3$
and a right divisor~$\pi_4$ of norm~$p$. We also see that if
$\alpha$ or $\beta$ is divisible by~$p$, then $\pi_1=\pi_3$
\emph{and} $\pi_2=\pi_4$ can be achieved, so the statement
of the lemma holds both on the left and on the right.

Thus we can assume that $\alpha$ and $\beta$ are not
divisible by~$p$. By Lemma~\ref{twin_char}, we have
$\alpha\konj{\beta}=-\beta\konj{\alpha}$. Suppose first that
this quaternion is not divisible by~$p$. The uniqueness
statement of Lemma~\ref{quat_lnko} can be applied to
$\alpha\konj{\beta}=-\beta\konj{\alpha}$, so $\pi_1$ and
$\pi_3$ are right associates, and the statement of the lemma
holds on the left.

If $\alpha\konj{\beta}=-\beta\konj{\alpha}$ is divisible
by~$p$, then apply Lemma~\ref{twin_primeprop} to
$\pi=\konj{\pi_2}$, $\theta=\konj{\alpha}$,
$\eta=\konj{\beta}$. We get that
$\konj{\pi_2}\mid\konj{\beta}$, so $\pi_2$ is a right
divisor of $\beta$, too, and the statement of the lemma
holds on the right.
\end{proof}

\begin{lm}\label{icube_param}
Let $\alpha_1,\ldots,\alpha_m\in\BE$ be pairwise twins and
$p\in\BZ$ a prime dividing their common norm. Then there
exists a quaternion $\pi\in\BE$ with norm~$p$ such that
either $\pi$ divides every $\alpha_\ell$ on the left, or
$\pi$ divides every $\alpha_\ell$ on the right.
\end{lm}

\begin{proof}
Lemma~\ref{quat_lnko} yields a left divisor $\pi_\ell$ of
$\alpha_\ell$ and a right divisor $\rho_\ell$ of
$\alpha_\ell$ with norm~$p$. If $p$ does not divide
$\alpha_\ell$, then $\pi_\ell$ and $\rho_\ell$ are
essentially unique, otherwise they can be chosen
arbitrarily. Thus we can disregard those $\alpha_\ell$ that
are divisible by~$p$, and can assume (to simplify notation)
that no $\alpha_\ell$ is divisible by~$p$.

Consider the complete graph on $\{1,2,\ldots,m\}$. Color the
edge $\{u,v\}$ to Lilac if $\pi_u$ and $\pi_v$ are right
associates, and to Red if $\rho_u$ and $\rho_v$ are left
associates (any edge can carry both colors). The previous
lemma shows that every edge has a color. By uniqueness, the
lilac edges, as well as the red edges yield a transitive
relation. This implies by an elementary graph-theoretic
argument that either every edge is lilac, or every edge is
red.
\end{proof}

Recall that $Q=\{\pm 1, \pm i, \pm j, \pm k\}$. We
characterize icubes in~$\BE$ first.

\begin{thm}\label{twin_decomp_thm}
Let $(\alpha_1,\ldots,\alpha_m)$ be an $m$-icube
$\in\BE$. Then there exist $\gamma,\delta\in\BE$ and an
$m$-icube $(\vep_1,\ldots,\vep_m)\in Q^m$ such that
$\vep_1=1$ and $\alpha_\ell=\gamma\vep_\ell\delta$ for every
$1\le\ell\le m$. Conversely, every such
$(\gamma\vep_1\delta,\ldots,\gamma\vep_m\delta)$ is an
$m$-icube in~$\BE$.
\end{thm}

\begin{proof}
Corollary~\ref{twin_char_cor} implies that
$(\gamma\vep_1\delta,\ldots,\gamma\vep_m\delta)$ is an
$m$-icube. Conversely, suppose that
$C=(\alpha_1,\ldots,\alpha_m)$ is an $m$-icube in~$\BE$.
Applying Lemma~\ref{icube_param} and
Corollary~\ref{twin_char_cor} several times successively we
see that $C=(\gamma\vep_1\delta,\ldots,\gamma\vep_m\delta)$
for some $\gamma,\delta\in\BE$ and units
$\vep_\ell\in\BE$. Replacing $\gamma$ by $\gamma\vep_1$ we
can assume that $\vep_1=1$. Then $(\vep_1,\ldots,\vep_m)$ is
an $m$-icube by Corollary~\ref{twin_char_cor}. Since $1$ and
$\vep_\ell$ are twins, the real part of each $\vep_\ell$ is
zero for $\ell\ge 2$, and therefore $\vep_\ell$ has integer
coefficients.
\end{proof}

To prove extension results we have to characterize icubes
in~$\BL$. To count them, we need uniqueness in the above
decomposition. To achieve these ends we reduce the problem
to quaternions with odd norm. The next statement is clear by
Claim~\ref{twin_even}.

\begin{clm}\label{odd_reduction}
  Suppose that $N=2^nD$ where $n\ge 2$ and $D$ is odd. Then
  every $m$-icube $(\alpha_1,\ldots,\alpha_m)$ in\/~$\BL$
  with edge norm $N$ can be written uniquely as
\[
\big((1+i)^{n-1}\beta_1,\ldots,(1+i)^{n-1}\beta_m\big)\,,
\]
where $(\beta_1,\ldots,\beta_m)$ is also an $m$-icube
in\/~$\BL$. Thus $f_m(N)=f_m(2D)$.\qed
\end{clm}

\begin{clm}\label{odd_reduction_2}
Let $N=2D$ where $D$ is odd. Then every $m$-icube
$(\alpha_1,\ldots,\alpha_m)$ in\/~$\BL$ with edge norm $N$
can be written uniquely as
$(\eta\beta_1,\ldots,\eta\beta_m)$, where
$(\beta_1,\ldots,\beta_m)$ is an $m$-icube in\/~$\BL$ and
$\eta\in\{1+i, 1+j, 1+k\}$. Therefore $f_m(N)=3f_m(D)$.
\end{clm}

\begin{proof}
By the proof of Claim~\ref{twin_even} we see that exactly
two of the components of every~$\alpha_\ell$ are even. Since
the vectors corresponding to $\alpha_1,\ldots,\alpha_m$ are
pairwise orthogonal, looking at the scalar products
modulo~$2$ we see that these two-element subsets of the
indices are either equal or disjoint. Thus $(1)$ of
Claim~\ref{twin_even} shows that the same element of $\{1+i,
1+j, 1+k\}$ can be pulled out of each $\alpha_i$ on the
left. This shows that every $m$-icube of edge norm $D$
yields exactly three $m$-icubes of edge norm~$2D$.
\end{proof}

To each $m$-icube $(\alpha_1,\ldots,\alpha_m)$ with odd edge
norm assign the unique sequence $(g_1,\ldots,g_m)$ with the
property that $\alpha_\ell\in S_{g_\ell}$ for every $1\le
\ell\le m$ (see Claim~\ref{Sg_elementary}). This sequence is
called the \emph{type} of $(\alpha_1,\ldots,\alpha_m)$. By
Claim~\ref{Sg_elementary}~(1), the elements $g_\ell\in K$
are pairwise different. Call an $m$-icube
$(\alpha_1,\ldots,\alpha_m)$ \emph{orderly}, if its type is
$(1)$ or $(1,i)$, or $(1,i, j)$, or $(1,i, j, k)$, depending
on $m$.

Permuting the components of vectors preserves norm as well
as orthogonality. If $(g_1,\ldots,g_m)$ and
$(h_1,\ldots,h_m)$ are types, then one can fix a permutation
$r$ of $K$ that maps each $h_\ell$ to $g_\ell$. This
permutation induces a bijection on~$\BL$:
\[
\alpha=a_1+a_ii+a_jj+a_kk\mapsto 
r(\alpha)=a_{r(1)}+a_{r(i)}i+a_{r(j)}j+a_{r(k)}k\,.
\]
If $(\alpha_1,\ldots,\alpha_m)$ has type $(g_1,\ldots,g_m)$,
then $\big(r(\alpha_1),\ldots,r(\alpha_m)\big)$ has type
$(h_1,\ldots,h_m)$, so the number of $m$-icubes of type
$(g_1,\ldots,g_m)$ does not depend on
$(g_1,\ldots,g_m)$. The number of possible types is $4\cdot
3\cdot\ldots\cdot (4-m+1)=24/(4-m)!$. Hence:

\begin{clm}\label{odd_reduction_1}
Let $N$ be odd. Then $f_m(N)=24M/(4-m)!$, where $M$ is the
number of orderly $m$-icubes with edge norm~$N$.\qed
\end{clm}

\begin{thm}\label{twin_decomp_odd}
Let $\gamma$ and $\delta\in\BL$ be primary quaternions, and
$\vep_1=\pm 1$, $\vep_2=\pm i$, $\vep_3=\pm j$ and
$\vep_4=\pm k$. Then
$(\gamma\vep_1\delta,\ldots,\gamma\vep_m\delta)$ is an
orderly $m$-icube in\/~$\BL$. Conversely, every orderly
$m$-icube in\/~$\BL$ with odd edge norm can be obtained this
way.
\end{thm}

\begin{proof}
Since $\gamma$ and $\delta$ are primary,
$(\gamma\vep_1\delta,\ldots,\gamma\vep_m\delta)$ is orderly
by $(4)$ of Claim~\ref{Sg_elementary}. Conversely, suppose
that $C$ is an orderly $m$-icube in\/~$\BL$ with odd
edge norm. Apply Lemma~\ref{icube_param} successively, but
in every step make sure that $\pi$ is primary (this can be
done by
Claim~\ref{Sg_elementary2}). Claim~\ref{Sg_elementary}
ensures that after pulling out~$\pi$ we get an orderly icube
in\/~$\BL$. Thus we get a representation
$C=(\gamma\vep_1\delta,\ldots,\gamma\vep_m\delta)$, where
$\gamma$ and~$\delta$ are primary. Since
$(\vep_1,\ldots,\vep_m)$ is orderly, $\vep_1=\pm 1$,
$\vep_2=\pm i$, $\vep_3=\pm j$ and $\vep_4=\pm k$.
\end{proof}

Now we prove the extension property
(Theorem~\ref{main_extension}). Suppose that an $m$-icube
$C$ in\/~$\BL$ is given. Claims~\ref{odd_reduction} and
\ref{odd_reduction_2} show that we can write $C$ as
$(\eta\beta_1,\ldots,\eta\beta_m)$, where
$(\beta_1,\ldots,\beta_m)$ is an $m$-icube in\/~$\BL$ with
odd edge norm. By rearranging the coordinates of the vectors
we get an $m$-icube of the form
$(\gamma\vep_1\delta,\ldots,\gamma\vep_m\delta)$ by
Theorem~\ref{twin_decomp_odd}. This clearly extends to an
$m+1$-icube
$(\gamma\vep_1\delta,\ldots,\gamma\vep_{m+1}\delta)$. Permuting
the coordinates back, and then multiplying by $\eta$ we get
the desired extension of~$C$.

\section{Counting}
\label{sec_count}

To deal with the case $m=3$ and $m=4$ we prove uniqueness in
Theorem~\ref{twin_decomp_odd}. We keep the notation that
$\vep_1=\pm 1$, $\vep_2=\pm i$, $\vep_3=\pm j$ and
$\vep_4=\pm k$.

\begin{lm}\label{twin_unique}
Suppose that
$\gamma_1\vep_\ell\delta_1=\gamma_2\vep_\ell\delta_2$ for
$\ell=u,v$ and
$\Norm(\gamma_1\delta_1)=\Norm(\gamma_2\delta_2)\ne 0$.
Then $\konj{\gamma_1}\gamma_2$ permutes with
$\vep_u\konj{\vep_v}$.
\end{lm}

\begin{proof}
Multiply the first equation by the conjugate of the
second. We obtain that
$\Norm(\delta_1)\gamma_1\vep_u\konj{\vep_v}\;\konj{\gamma_1}=
\Norm(\delta_2)\gamma_2\vep_u\konj{\vep_v}\;\konj{\gamma_2}$. Now
multiply on the left by~$\konj{\gamma_1}$ and on the right
by~$\gamma_2$, and then simplify by
$\Norm(\gamma_1\delta_1)=\Norm(\gamma_2\delta_2)$.
\end{proof}

An $m$-icube is called \emph{primitive}, if the $4m$
components of its $m$ vectors have no common divisor other
than~$\pm 1$.

\begin{thm}\label{icube_unique}
Suppose that $m\ge 3$ and
$(\gamma_1\vep_1\delta_1,\ldots,\gamma_1\vep_m\delta_1)=
(\gamma_2\vep_1\delta_2,\ldots,\gamma_2\vep_m\delta_2)$ are
two representations of an icube given by
Theorem~\ref{twin_decomp_odd}, where $\gamma_1$
and~$\gamma_2$ are primitive. Then $\gamma_1=\gamma_2$ and
$\delta_1=\delta_2$.  Furthermore, such an icube
$(\gamma_1\vep_1\delta_1,\ldots,\gamma_1\vep_m\delta_1)$ is
primitive if and only if $\gamma_1$ and $\delta_1$ are both
primitive.
\end{thm}

\begin{proof}
Lemma~\ref{twin_unique} shows that $\konj{\gamma_1}\gamma_2$
permutes with $\vep_1\konj{\vep_2}=\pm i$ and with
$\vep_1\konj{\vep_3}=\pm j$. Therefore
$d=\konj{\gamma_1}\gamma_2$ is a real number. As $\gamma_1$
and $\gamma_2$ are primary, $d\in\BZ$. We have
$d\gamma_1=\Norm(\gamma_1)\gamma_2$. Since $\gamma_1$ and
$\gamma_2$ are primitive, the gcd of the coefficients of the
two sides of this equation is $\Norm(\gamma_1)=\pm
d$. Therefore $\gamma_2=\pm\gamma_1$. Since they are
primary, they are equal. Then
$\gamma_1\vep_1\delta_1=\gamma_2\vep_1\delta_2$
yields~$\delta_1=\delta_2$.

Now let
$C=(\gamma_1\vep_1\delta_1,\ldots,\gamma_1\vep_m\delta_1)$
such that $\gamma_1$ and $\delta_1$ are primitive and assume
to get a contradiction that $C$ is not primitive. Write $C$
as $cC'$, where $c>1$ and $C'$ is a primitive $m$-icube.
Then $C'$ can be represented as
$(\gamma_3\vep_1\delta_3,\ldots,\gamma_3\vep_m\delta_3)$,
where $\gamma_3$ must be primitive, so $C$ has a
representation $C=\big(\gamma_3\vep_1(c\delta_3),\ldots,
\gamma_3\vep_m(c\delta_3)\big)$. Here $c\delta_3$ is also
primary, since $c$ is a positive odd integer. The uniqueness
statement proved in the previous paragraph shows that
$\delta_1=c\delta_3$, contradicting the assumption that
$\delta_1$ is primitive.
\end{proof}

\begin{cor}\label{icube_orderly_count}
Suppose that $m\ge 3$ and $N$ is an odd integer. Then the
number of orderly, primitive $m$-icubes with edge norm~$N$
is
\[
k(N)=2^m\sum_{d\mid N} h(d)h(N/d)\,.
\]
Here $h(d)=d\prod_p \big(1+(1/p)\big)$, where $p$ runs over
the prime divisors of $d$.
\end{cor}

\begin{proof}
Recall that $h(d)$ is the number of primitive, primary
quaternions with norm~$d$ by
Claim~\ref{twin_primary_count}. Consider the unique
representation given by Theorem~\ref{icube_unique}, and
let $d=\Norm(\gamma_1)$. Then $\Norm(\delta_1)=N/d$. These
two quaternions can be chosen $h(d)h(N/d)$ ways, and $d$ can
be any divisor of~$N$. Finally, there are $2^m$
possibilities to chose the signs of $\vep_1,\ldots,\vep_m$.
\end{proof}

We can now compute $f_4$ as stated in
Theorem~\ref{main_count}. Fix $m=4$ and write $N=2^nD$,
where $D$ is odd. By Claims \ref{odd_reduction} and
\ref{odd_reduction_2} we have that $f_4(N)=3f_4(D)$ if $n\ge
1$. Next Claim~\ref{odd_reduction_1} shows that
$f_4(D)=24M$, where $M$ is the number of orderly $m$-icubes
with edge norm~$D$. Finally, each orderly $m$-icube can be
written uniquely as $C=cC'$, where $c$ is a positive integer
and $C'$ is primitive. Clearly, $c\mid D^2$, and therefore
\[
M=\sum_{c^2\mid N} k(N/c^2)\,,
\]
where $k$ is the function defined in
Corollary~\ref{icube_orderly_count} for $m=4$. Thus
\[
f_4(D)=(16\cdot 24)\sum_{c^2\mid N} 
\sum_{d\,\mid\, (N/c^2)}h(d)h\big(N/(c^2d)\big)\,.
\]
It is a well-known fact that the convolution of
multiplicative functions is multiplicative. Since $h$ is
obviously multiplicative, so is $k/16$, which is the
convolution of $h$ by itself. The function assigning $1$ to
squares and $0$ to all other integers is also
multiplicative, so the double sum above (which is
$f_4(D)/384$) is also multiplicative for odd values
of~$D$. Finally the remarks at the beginning of this
argument show that $f_4(N)/384$ is multiplicative on the set
of positive integers.

The proof above clearly shows that $f_4(2^n)=384\cdot 3$
for~$n\ge 1$. If $p$ is an odd prime, then it is a routine
calculation to prove, using the last displayed formula, that
the value of $f_4(p^n)$ is the one given in $(2)$ of
Theorem~\ref{main_count}. This somewhat complicated
summation is left to the reader.

To show that $f_3(N)=f_4(N)/2$ one can either go through the
argument above with $m=3$, or invoke
Theorem~\ref{main_extension_gen} (stating that each
$3$-icube has exactly two extensions in dimension~$4$). To
compute $f_2(N)$ we need to improve
Theorem~\ref{twin_decomp_odd}, since uniqueness does not
hold for $m=2$.

\begin{thm}\label{icube_unique2}
Every orderly $2$-icube in\/~$\BL$ with odd edge norm can be
written in the from
$(\gamma\vep_1\delta,\gamma\vep_2\delta)$, where $\gamma$
and $\delta\in\BL$ are primary quaternions such that
$\gamma\,i\,\konj{\gamma}$ is primitive and $\vep_1=\pm 1$,
$\vep_2=\pm i$. Here $\gamma$ and $\delta$ are uniquely
determined. Such an icube is primitive if and only if
$\delta$ is primitive as well.
\end{thm}

\begin{proof}
Theorem~\ref{twin_decomp_odd} yields a decomposition
$C=(\gamma\vep_1\delta,\gamma\vep_2\delta)$. Suppose that
$\gamma\,i\,\konj{\gamma}$ is divisible by a
prime~$p$. Apply Lemma~\ref{twin_param} to $\alpha=\gamma$
and $\beta=\gamma i$. We get that there is a primary $\pi$
with norm~$p$ that divides both $\gamma$ and $\gamma i$ on
the \emph{right.} Let $\gamma=\gamma_1\pi$ (so~$\gamma_1$ is
primary). Now $\gamma i=\gamma_1\pi i$ is right divisible by
$\pi$, so the uniqueness statement of Lemma~\ref{quat_lnko}
shows that $\pi i=\vep \pi$ for some unit~$\vep$. As $\pi$
is primary, Claim~\ref{Sg_elementary2} gives that $\vep\in
Q$, so $\vep\in S_i$ by Claim~\ref{Sg_elementary}, that is,
$\vep=\pm i$. Therefore we can write $C$ as
$\big(\gamma_1\vep_1(\pi\delta),
\gamma_1(\pm\vep_2)(\pi\delta)\big)$. Applying this several
times we get a representation where
$\gamma\,i\,\konj{\gamma}$ is primitive.

Suppose that $C=(\alpha_1,\alpha_2)=
(\gamma_1\vep_1\delta_1,\gamma_1\vep_2\delta_1)=
(\gamma_2\vep_1\delta_2,\gamma_2\vep_2\delta_2)$ are two
representations such that
$\gamma_\ell\,i\,\konj{\gamma_\ell}$ are both
primitive. Then $\alpha_2\konj{\alpha_1}=
\Norm(\delta_\ell)\gamma_\ell(\vep_2\konj{\vep_1})
\konj{\gamma_\ell}$. Here $\vep_2\konj{\vep_1}=\pm i$, and
therefore this quaternion determines $\Norm(\delta_\ell)$ as
the positive gcd of its coefficients. Thus
$\Norm(\delta_1)=\Norm(\delta_2)$ and
$\gamma_1\,i\,\konj{\gamma_1}=
\gamma_2\,i\,\konj{\gamma_2}$. The uniqueness statement of
Lemma~\ref{3repr} shows that $\gamma_1=\gamma_2$, since both
are primary. Thus $\delta_1=\delta_2$ as well.  The
uniqueness statement in the last sentence of the theorem can
be proved exactly as in Theorem~\ref{icube_unique}.
\end{proof}

\begin{cor}\label{icube_orderly_count2}
Suppose that $N$ is an odd integer. Then the number of
orderly, primitive $2$-icubes with edge norm~$N$ is
\[
k_2(N)=4\sum_{d\mid N} q(d)h(N/d)\,,
\]
where $p$ runs over the prime divisors of $d$ and the
functions $q$ and $h$ are given by Claim~\ref{twin_count2}
and Claim~\ref{twin_primary_count}, respectively.
\end{cor}

\begin{proof}
Consider the unique representation given by
Theorem~\ref{icube_unique2}, and let
$d=\Norm(\gamma_1)$. Then $\Norm(\delta_1)=N/d$. These two
quaternions can be chosen $q(d)h(N/d)$ ways by the claims
quoted in the corollary, and $d$ can be any divisor
of~$N$. Finally, there are $4$ possibilities to chose the
signs of $\vep_1$ and $\vep_2$.
\end{proof}

\vbox{%
To compute $f_2$ as stated in Theorem~\ref{main_count} we
mimic the argument presented above for~$f_4$. The reduction
to odd norm is the same, and if $D$ is odd, then we get that
\[
f_2(D)=(4\cdot 12)\sum_{c^2\mid N} 
\sum_{d\,\mid\, (N/c^2)}q(d)h\big(N/(c^2d)\big)\,.
\]
Again, the details of the summation are left to the reader.}

\end{document}